\documentclass{amsart}

\usepackage{mathrsfs}
\usepackage{mathrsfs}
\usepackage{amsfonts}
\usepackage{amsfonts}
\usepackage{mathrsfs}
\usepackage{amsfonts}
\usepackage{amsfonts}
\usepackage{amsfonts}
\usepackage{mathrsfs}
\usepackage{amsfonts}
\usepackage{amssymb,amsmath}
\usepackage{amsthm}
\usepackage{amsfonts}

\newtheorem{theorem}{Theorem}[section]
\newtheorem{lemma}[theorem]{Lemma}

\theoremstyle{definition}
\newtheorem{definition}[theorem]{Definition}

\theoremstyle{remark}
\newtheorem{remark}[theorem]{Remark}

\numberwithin{equation}{section}

\theoremstyle{problem}

\begin{document}

\title[Variable order space-fractional diffusion equation]
 {A Carleman type estimate of variable order space-fractional diffusion equations and applications to some inverse problems}

\author[J.X.Jia]{Junxiong Jia}
\address{School of Mathematics and Statistics,
Xi'an Jiaotong University,
 Xi'an
710049, China; Beijing Center for Mathematics and Information Interdisciplinary Sciences (BCMIIS);}
\email{jjx323@xjtu.edu.cn}

\author[J. Peng]{Jigen Peng}
\address{School of Mathematics and Statistics,
Xi'an Jiaotong University,
 Xi'an
710049, China; Beijing Center for Mathematics and Information Interdisciplinary Sciences (BCMIIS);}
\email{jgpengxjtu@126.com}






\begin{abstract}
Variable order space-fractional diffusion equation derived as an important model to describe anomalous diffusion phenomenon.
In this article, well-posedness has been proved for equations with the ``Dirichlet'' or the ``Neumann''
type volume-constrained conditions by using the technique of the nonlocal vector calculus.
Then some regularity properties have been obtained under the variable-order Sobolev space framework.
By choosing a space-independent weight function and using the technique of the nonlocal vector calculus,
a Carleman type estimate has been obtained. At last, the Carleman type estimate has been used both to a backward diffusion problem
and an inverse source problem to obtain some uniqueness and stability results.
\end{abstract}

\maketitle


\section{Introduction}

Nonlocal diffusion equations have been wildly studied for its important applications in
physics \cite{Metzler20001}, image processing \cite{BaiIEEE2007} and so on.
From the mathematical point of view, there are also a lot of investigations \cite{andreu2010nonlocal,Jia2014605,Li2012476,SAKAMOTO2011426}.
In this paper, we focus on a special type of nonlocal diffusion equations that is
the variable order space-fractional diffusion equations. In a simple setting,
the model we studied has the following form
\begin{align}\label{simpleEquation}
\partial_{t}u(x,t) + (-\Delta)^{\beta(x)}u(x,t) = f(x,t) \quad \text{in }\Omega\times(0,T),
\end{align}
accompanied with appropriate initial and boundary conditions. Here $\beta(\cdot)$ is a continuous function between $0$ and $1$.
Actually, we study a general model with anisotropic diffusion coefficient but, in this section, we only present this simple model for
the general model could not be clearly shown in a few words.

When $\beta(\cdot)$ is a constant, a lot of results have been obtained recently.
Du, Gunzburger, Lehoucq and Zhou \cite{du2012analysis, DuQiang2013MMMAS} develop the technique of the nonlocal vector calculus and
construct weak solutions for some general nonlocal equations.
In 2013, a Harnack's inequality has been obtained by Felsinger and Kassmann \cite{FelKass2013}
for a general space nonlocal diffusion equation which include equation (\ref{simpleEquation}) with constant $\beta$.
In 2015, Felsinger, Kassmann and Voigt \cite{Felsinger2015} construct weak solutions for
general nonlocal equations which include non-sysmetric kernels.
In 2016, Fern\'{a}ndez-Real and Ros-Oton \cite{fernandez2014boundary} prove some results about the boundary regularity,
which have been generalized to a very general setting in a recent paper \cite{fernandez2015regularity}.

When $\beta(\cdot)$ is a function, to the best of our knowledge, the studies seem to be limited.
When $\beta(\cdot)$ is Lipschitz continuous, Tscuchiya \cite{Masaaki1992} study a stochastic differential equation related to $(-\Delta)^{\beta(\cdot)}$.
In a recent book \cite{applebaum2009levy}, there are some studies from the operator semigroup perspective.
In \cite{silvestre2006holder}, Silvestre study the H\"{o}lder regularity of the elliptic equations with a
variable order fractional Laplace operator.

Because the studies of diffusion equations with variable order space-fractional Laplace operator are little,
there are also rare studies on the related inverse problems. However, for constant $\beta$ case,
there are already some investigations.
Uniqueness and some numerical results have been obtained in
\cite{Ramazan2016Simultaneous,tatar2015determination,tatar2013uniqueness,tatar2015inverse}
for equations with fractional Laplace operator.
For the diffusion equation with the fractional Laplace operator on a periodic domain,
backward diffusion problem has been studied in \cite{jia2015bayesian} under the Bayesian statistical framework and
the same backward diffusion problem has also been studied by using the
variable total variation regularization method in \cite{jia2016variable}.

In this paper, we firstly construct a weak solution, then an improved regularity result has been obtained.
Using a space-independent two parameter weight function and the first Green's identity for the
general nonlocal operators, we prove a Carleman type estimate.
Let us recall that the Carleman type estimates concerned with the integer-order parabolic equations are useful tools
in studying inverse problems e.g. \cite{Jishan2009Inverse,Oleg1998Inverse,BinWu2012Inverse,yamamoto2009carleman}.
However, for the fractional diffusion equations, there seems only one paper \cite{XiangXu2011Applicable} provide some results on
one-dimensional time-fractional space-integer order diffusion equations.
Hence, our result on the Carleman type estimate may be a useful tool in studying inverse problems related
to the space-fractional diffusion equations.
In the last part of this paper, we provide two applications of the Carleman type estimate.
Specifically speaking, a stability result for backward diffusion problem
has been proved and a uniqueness result for an inverse source problem has also been obtained.

The organization of this paper is as follows.
In Section \ref{nonlocalEquation}, some equivalence relation of function space with variable derivative
have been proved. Then, by using Lax-Milgram theorem and Galerkin-type arguments, well-posedness
results have been constructed for variable order space-fractional diffusion equation.
In the last part of Section \ref{nonlocalEquation}, some weak regularity properties have been studied for
our purpose on the inverse problems.
In Section \ref{CarlemanEstimateSection}, a Carleman type estimate has been constructed, which provide a powerful
tool for studying inverse problems.
In Section \ref{BackwardSection}, the stability of a backward diffusion problem has been obtained by using the
Carleman type estimate. As another application, an inverse source problem has also been studied.

\section{Nonlocal equations}\label{nonlocalEquation}

In this section, a short review of some important concepts in nonlocal vector calculus will be provided, then
we propose our general variable order space-fractional diffusion equation with the ``Dirichlet'' and
the ``Neumann'' type volume constrains.
After that, well-posedness and regularity properties will be studied which provide a foundation for
our investigations on some Carleman type estimates. Before going further, let us provide some notations used in
all of this paper:
\begin{itemize}
  \item $n$ denotes the space dimension; $\Omega \subset \mathbb{R}^{n}$ is a bounded open set;
  \item $L^{p}(\Omega)$ ($1\leq p<\infty$) denotes the usual Lebesgue integrable function space;
  $W^{s,p}(\Omega)$ denotes the usual Sobolev space for functions with $s$-order derivative belong to $L^{p}(\Omega)$;
  If $p=2$, we denote $H^{s}(\Omega)$ as $W^{s,2}(\Omega)$;
  \item For a function $u$ depend on the time variable $t$, $u'$ stands for the time derivative of $u$;
  \item Denote $A$ as a matrix, then $A^{T}$ stands for the transpose of the matrix $A$;
\end{itemize}

\subsection{A brief review of the nonlocal vector calculus}

In \cite{du2012analysis, DuQiang2013MMMAS}, a nonlocal vector calculus is developed.
Here, we briefly review the aspects of that calculus that are useful in what follows.

Given vector mapping $\nu(x,y), \alpha(x,y)\,:\, \mathbb{R}^{n}\times\mathbb{R}^{n}\rightarrow\mathbb{R}^{k}$
with $\alpha$ antisymmetric, i.e., $\alpha(y,x) = -\alpha(x,y)$, the action of the nonlocal divergence operator $\mathcal{D}$
on $\nu$ is defined as
\begin{align}\label{nonlocalDivergenceDef1}
\mathcal{D}(\nu)(x) := \int_{\mathbb{R}^{n}}(\nu(x,y) + \nu(y,x))\cdot\alpha(x,y)dy \quad \text{for }x\in\mathbb{R}^{n},
\end{align}
where $\mathcal{D}(\nu) \,:\, \mathbb{R}^{n}\rightarrow\mathbb{R}$.

Given the mapping $u(x)\,:\, \mathbb{R}^{n}\rightarrow\mathbb{R}$, the adjoint operator $\mathcal{D}^{*}$ corresponding to $\mathcal{D}$
is the operator whose action on $u$ is given by
\begin{align}\label{nonlocalDivergenceDef2}
\mathcal{D}^{*}(u)(x,y) = - (u(y)-u(x))\alpha(x,y) \quad \text{for }x,y\in \mathbb{R}^{n},
\end{align}
where $\mathcal{D}^{*}(u) \,:\, \mathbb{R}^{n}\times\mathbb{R}^{n}\rightarrow\mathbb{R}^{k}$.
With $\mathcal{D}^{*}$ denoting the adjoint of the nonlocal divergence operator, we view $-\mathcal{D}^{*}$ as
a nonlocal gradient.

From (\ref{nonlocalDivergenceDef1}) and (\ref{nonlocalDivergenceDef2}), one easily deduces that if $a(t,x,y) = a(t,y,x)$
denotes a second-order tensor satisfying $a = a^{T}$, then
\begin{align*}
\mathcal{D}(a\cdot\mathcal{D}^{*}u)(x) = -2 \int_{\mathbb{R}^{n}}(u(y)-u(x))\alpha(x,y)\cdot(a(t,x,y)\cdot\alpha(x,y))dy
\quad \text{for }x\in \mathbb{R}^{n},
\end{align*}
where $\mathcal{D}(a\cdot\mathcal{D}^{*}u)\,:\, \mathbb{R}^{n}\rightarrow\mathbb{R}$.
In the following, sometimes, we denote $\gamma(t,x,y) := \alpha(x,y)\cdot(a(t,x,y)\cdot\alpha(x,y))$.
Because $\mathcal{D}$ and $\mathcal{D}^{*}$ are adjoint operators, if $a$ is also positive definite, the operator
$\mathcal{D}(a\cdot\mathcal{D}^{*}\cdot)$ is nonnegative.

Given an open subset $\Omega\subset\mathbb{R}^{n}$, the corresponding interaction domain is defined as
\begin{align}\label{nonlocalDivergenceDef3}
\Omega_{\mathcal{I}} := \left\{
y\in\mathbb{R}^{n}\backslash\Omega \quad \text{such that }\alpha(x,y) \neq 0 \text{ for some }x\in\Omega
\right\},
\end{align}
so that $\Omega_{\mathcal{I}}$ consists of those points outside of $\Omega$ that interact with points in $\Omega$.
Note that the situation $\Omega_{\mathcal{I}} = \mathbb{R}^{n}\backslash\Omega$ is allowable, as is $\Omega = \mathbb{R}^{n}$,
in which case $\Omega_{\mathcal{I}} = \emptyset$.
Then, corresponding to the divergence operator $\mathcal{D}(\nu)\,:\,\mathbb{R}^{n}\rightarrow\mathbb{R}$ defined
in (\ref{nonlocalDivergenceDef1}), we define the action of the nonlocal interaction operator
$\mathcal{N}(\nu)\,:\, \mathbb{R}^{n}\rightarrow\mathbb{R}$ on $\nu$ by
\begin{align}\label{nonlocalDivergenceDef4}
\mathcal{N}(\nu)(x) := -\int_{\Omega\cup\Omega_{\mathcal{I}}}(v(x,y)+\nu(y,x))\cdot\alpha(x,y)dy
\quad \text{for }x\in\Omega_{\mathcal{I}}.
\end{align}
It is shown in \cite{du2012analysis, DuQiang2013MMMAS} that $\int_{\Omega_{\mathcal{I}}}\mathcal{N}(\nu)dx$ can be
viewed as a nonlocal flux out of $\Omega$ into $\Omega_{\mathcal{I}}$.

With $\mathcal{D}$ and $\mathcal{N}$ defined as in (\ref{nonlocalDivergenceDef1}) and (\ref{nonlocalDivergenceDef4}), respectively,
we have the nonlocal Gauss theorem
\begin{align}\label{nonlocalDivergenceDef5}
\int_{\Omega}\mathcal{D}(\nu)dx = \int_{\Omega_{\mathcal{I}}}\mathcal{N}(\nu)dx.
\end{align}

Next, let $u(x)$ and $v(x)$ denote scalar functions. Then it is a simple matter to show that the nonlocal
divergence formula (\ref{nonlocalDivergenceDef5}) implies the generalized nonlocal Green's first identity
\begin{align}\label{nonlocalGreenFirst}
\int_{\Omega}v\mathcal{D}(a\cdot\mathcal{D}^{*}u)dx -
\int_{\Omega\cap\Omega_{\mathcal{I}}}\int_{\Omega\cap\Omega_{\mathcal{I}}}\mathcal{D}^{*}v\cdot(a\cdot \mathcal{D}^{*}u)dydx
= \int_{\Omega_{\mathcal{I}}}v\mathcal{N}(a\cdot\mathcal{D}^{*}u)dx.
\end{align}

\subsection{Nonlocal diffusion equations}

In this subsection, we will describe our model precisely which include equation (\ref{simpleEquation}) as a special case.
With the nonlocal vector calculus, space nonlocal diffusion equation with the ``Dirichlet'' volume-constrained condition can be written as
\begin{align}\label{Equation1}
\left\{\begin{aligned}
\partial_{t}u + \mathcal{D}(a\cdot \mathcal{D}^{*}u) & = f(t,x) \quad && \text{on }\Omega \times (0,T),  \\
u(t,x) & = 0 && \text{on }\Omega_{\mathcal{I}} \times (0,T), \\
u(0,x) & = u_{0}(x) && \text{on }\Omega\cup\Omega_{\mathcal{I}}.
\end{aligned}\right.
\end{align}
Space nonlocal diffusion equation with the ``Neumann'' volume-constrained condition could be written as
\begin{align}\label{Equation2}
\left\{\begin{aligned}
\partial_{t}u + \mathcal{D}(a\cdot \mathcal{D}^{*}u) & = f(t,x) \quad && \text{on }\Omega \times (0,T),  \\
\mathcal{N}(a\cdot\mathcal{D}^{*}u) & = 0 && \text{on }\Omega_{\mathcal{I}}\times(0,T),   \\
u(0,x) & = u_{0}(x) && \text{on }\Omega\cup\Omega_{\mathcal{I}},   \\
\int_{\Omega\cup\Omega_{\mathcal{I}}}udx & = 0. &&
\end{aligned}\right.
\end{align}

Let us specify the functions $\alpha(\cdot,\cdot)$, $a(\cdot,\cdot,\cdot)$ in the definitions of $\mathcal{D}$
and $\mathcal{D}^{*}$. Assume $\beta(x)$ to be a continuous function with upper bound $\beta^{*} < 1$
and lower bound $\beta_{*} > 0$. We choose
\begin{align}\label{defineAlpha}
\alpha(x,y) = \frac{y-x}{|y-x|^{n/2+\beta(x)+1}}1_{B_{\epsilon}(x)}(y)
\end{align}
and
\begin{align}\label{defineAlpha1}
a(t,x,y) = \left(a_{ij}(t,x,y)\right)_{1\leq i,j\leq n}
\end{align}
with $a_{ij}(t,x,y)\in C^{1}([0,T]\times\mathbb{R}^{n}\times\mathbb{R}^{n})$, $a_{ij} = a_{ji}$
and in addition, we assume
\begin{align}\label{aAssumption1}
\begin{split}
& 0 < a_{*}|\xi|^{2} \leq \sum_{i,j=1}^{n}a_{ij}(t,x,y)\xi_{i}\xi_{j} \leq a^{*} |\xi|^{2}, \\
& \quad\quad
\sum_{i,j=1}^{n}\partial_{t}a_{ij}(t,x,y)\xi_{i}\xi_{j} \leq a^{*} |\xi|^{2}
\end{split}
\end{align}
for all $\xi \in \mathbb{R}^{n}$ and $y \in B_{\epsilon}(x)$, where
$B_{\epsilon}(x):= \{ y\in\Omega\cup\Omega_{\mathcal{I}} \,:\, |y-x|\leq\epsilon \}$.
Denote $\gamma(t,x,y) := \alpha(x,y) \cdot a(t,x,y) \cdot \alpha(x,y)$, then we know that
\begin{align}\label{defineGamma}
\gamma(t,x,y) = \frac{(y-x)\cdot a(t,x,y)\cdot (y-x)}{|y-x|^{n+2\beta(x)+2}}1_{B_{\epsilon}(x)}(y).
\end{align}

\begin{remark}
If we choose $\epsilon = \infty$, then we easily know that $\Omega_{\mathcal{I}} = \mathbb{R}^{n}\backslash\Omega$.
Assume $a(t,x,y)$ to be the identity matrix, then the operator $\mathcal{D}(a\cdot\mathcal{D}^{*}\cdot)$ could be simplified as
$(-\Delta)^{\beta(\cdot)}$ which is appeared in equation (\ref{simpleEquation}).
Hence, equations (\ref{Equation1}) and (\ref{Equation2}) include equation (\ref{simpleEquation}) as a special case.
\end{remark}

As in \cite{du2012analysis}, we define the nonlocal energy norm, nonlocal energy space, and nonlocal volume-constrained energy space by
\begin{align}
& |||u||| := \left( \frac{1}{2}\int_{\Omega\cup\Omega_{\mathcal{I}}}\int_{\Omega\cup\Omega_{\mathcal{I}}}
\mathcal{D}^{*}(u)(x,y)\cdot(a(t,x,y)\cdot\mathcal{D}^{*}(u)(x,y)) dydx \right)^{1/2},  \label{nenrgyNorm} \\
& V(\Omega\cup\Omega_{\mathcal{I}}) := \{ u\in L^{2}(\Omega\cup\Omega_{\mathcal{I}}) \,:\, |||u||| < \infty \}, \label{nonlocalEnengySpace} \\
& V_{c}(\Omega\cup\Omega_{\mathcal{I}}) := \{ u\in V(\Omega\cup\Omega_{\mathcal{I}}) \,:\, E_{c}(u) = 0 \}, \label{nonlocalVolumeConstrained}
\end{align}
respectively, where $E_{c}(u)$ given by
\begin{align*}
E_{c}(u) := \int_{\Omega_{\mathcal{I}}} u^{2} dx \quad \text{if }\Omega_{\mathcal{I}} \neq \emptyset
\end{align*}
in the ``Dirichlet'' volume-constrained condition case and given by
\begin{align*}
E_{c}(u) := \left(\int_{\Omega\cup\Omega_{\mathcal{I}}} u dx\right)^{2} \quad \text{if }\Omega_{\mathcal{I}} \neq \emptyset
\end{align*}
in the ``Neumann'' volume-constrained condition case.
Define $|||u|||_{V_{c}^{*}(\Omega\cup\Omega_{\mathcal{I}})}$ to be the norm for the dual space $V_{c}^{*}(\Omega\cup\Omega_{\mathcal{I}})$
of $V_{c}(\Omega\cup\Omega_{\mathcal{I}})$ with respect to the standard $L^{2}(\Omega\cup\Omega_{\mathcal{I}})$ duality pairing.

For a Banach space $X$, we define
\begin{align}\label{defW}
W(0,T;X) := \left\{ u\in L^{2}(0,T;X)\,;\, u' \text{ exists and }u'\in L^{2}(0,T;X^{*}) \right\}.
\end{align}
$W(0,T;X)$ is a Banach space endowed with the norm
\begin{align*}
\|u\|_{W(0,T;X)}^{2} = \int_{0}^{T} \|u(t)\|_{X}^{2} dt + \int_{0}^{T} \|u'(t)\|_{X^{*}}^{2} dt.
\end{align*}
In addition, we denote
\begin{align}
L_{c}^{2}(\Omega\cup\Omega_{\mathcal{I}}) := \{ u\in L^{2}(\Omega\cup\Omega_{\mathcal{I}}) \,:\,
E_{c}(u) = 0 \}.
\end{align}

\begin{definition}\label{parabolicVariational}
(Parabolic variational formulation) Let $u_{0} \in L_{c}^{2}(\Omega\cup\Omega_{\mathcal{I}})$ and $f \in L^{2}(0,T; V_{c}^{*}(\Omega))$.
We say that $u \in W(0,T; V_{c}(\Omega\cup\Omega_{\mathcal{I}}))$ is a solution of (\ref{Equation1}) or (\ref{Equation2}) if for all
$\phi\in V_{c}(\Omega\cup\Omega_{\mathcal{I}})$
\begin{align*}
\int_{\Omega} \partial_{t}u \phi dx +
\int_{\tilde{\Omega}}\int_{\tilde{\Omega}}\mathcal{D}^{*}(u)\cdot(a(t,x,y)\cdot\mathcal{D}^{*}\phi)dydx
= & \int_{\Omega}f\phi dx, \quad \text{for a.e. }t\in(0,T), \\
u(0,x) = & u_{0}(x).
\end{align*}
\end{definition}

For convenience, we always denote $\tilde{\Omega}:=\Omega\cup\Omega_{\mathcal{I}}$,
$Q := \Omega\times(0,T)$ and $\tilde{Q} := \tilde{\Omega}\times(0,T)$ in the following parts of this article.

\subsection{Equivalence of spaces and well-posedness}

For a continuous function $\beta(\cdot)$ satisfying $0 < \beta_{*} \leq \beta(x) \leq \beta^{*} < 1$,
the variable-order Sobolev space is defined as
\begin{align}\label{defVariSobolev}
H^{\beta(\cdot)}(\Omega\cup\Omega_{\mathcal{I}}) := \left\{ u\in L^{2}(\Omega\cup\Omega_{\mathcal{I}}) \,:\,
\|u\|_{L^{2}(\Omega\cup\Omega_{\mathcal{I}})} + |u|_{H^{\beta(\cdot)}(\Omega\cup\Omega_{\mathcal{I}})} < \infty \right\},
\end{align}
where
\begin{align*}
|u|^{2}_{H^{\beta(\cdot)}(\Omega\cup\Omega_{\mathcal{I}})} := \int_{\Omega\cup\Omega_{\mathcal{I}}}\int_{\Omega\cup\Omega_{\mathcal{I}}}
\frac{(u(y)-u(x))^2}{|y-x|^{n+2\beta(x)}} dydx.
\end{align*}
Moreover, we define the subspaces
\begin{align}\label{defVariSobolev1}
H_{c}^{\beta(\cdot)}(\Omega\cup\Omega_{\mathcal{I}}) := \{ u\in H^{\beta(\cdot)}(\Omega\cup\Omega_{\mathcal{I}}) \,:\,
E_{c}(u) = 0 \}.
\end{align}

Now, let us provide some useful properties about the variable-order Sobolev space defined in (\ref{defVariSobolev}) and (\ref{defVariSobolev1}).
\begin{lemma}\label{embeddingVariableSpace}
Let $\beta(\cdot)$ to be a continuous function and there exist $\beta_{*}$ and $\beta^{*}$ such that
$0 < \beta_{*} \leq \beta(x) \leq \beta^{*} < 1$. Let $\Omega$, $\Omega_{\mathcal{I}}$ are open sets in $\mathbb{R}^{n}$,
and $u\,:\, \Omega\cup\Omega_{\mathcal{I}} \rightarrow \mathbb{R}$ be a measurable function. Then
\begin{align*}
C^{-1}\|u\|_{H^{\beta_{*}}(\Omega\cup\Omega_{\mathcal{I}})} \leq
\|u\|_{H^{\beta(\cdot)}(\Omega\cup\Omega_{\mathcal{I}})} \leq
C\|u\|_{H^{\beta^{*}}(\Omega\cup\Omega_{\mathcal{I}})}
\end{align*}
for some suitable positive constant $C = C(n,\beta_{*},\beta^{*}) \geq 1$. In particular,
\begin{align*}
H^{\beta^{*}}(\Omega\cup\Omega_{\mathcal{I}}) \subset H^{\beta(\cdot)}(\Omega\cup\Omega_{\mathcal{I}})
\subset H^{\beta_{*}}(\Omega\cup\Omega_{\mathcal{I}}).
\end{align*}
\end{lemma}
\begin{proof}
Denote $\tilde{\Omega}:=\Omega\cup\Omega_{\mathcal{I}}$. Firstly, we have
\begin{align}\label{embeddingGuji1}
\begin{split}
\int_{\tilde{\Omega}}\int_{\tilde{\Omega}\cap\{|y-x|\geq 1\}}\frac{|u(x)|^{2}}{|y-x|^{n+2\beta_{*}}}dxdy \leq &
\int_{\tilde{\Omega}}\left( \int_{|z|\geq 1}\frac{1}{|z|^{n+2\beta_{*}}} dx \right)|u(x)|^{2}dx \\
\leq & C \|u\|_{L^{2}(\Omega)}^{2},
\end{split}
\end{align}
where the integrability of the kernel $1/|z|^{n+2\beta_{*}}$ has been used.
If using $\beta^{*}$ instead of $\beta_{*}$, the above estimate also holds.
Because
\begin{align*}
\int_{\tilde{\Omega}}\int_{\tilde{\Omega}}\frac{|u(y)-u(x)|^{2}}{|y-x|^{n+2\beta_{*}}}dydx = &
\int_{\tilde{\Omega}}\int_{\tilde{\Omega}\cap\{|y-x|\geq 1\}}\frac{|u(y)-u(x)|^{2}}{|y-x|^{n+2\beta_{*}}}dydx \\
& + \int_{\tilde{\Omega}}\int_{\tilde{\Omega}\cap\{|y-x|< 1\}}\frac{|u(y)-u(x)|^{2}}{|y-x|^{n+2\beta_{*}}}dydx \\
= & \text{I}_{1} + \text{I}_{2},
\end{align*}
\begin{align*}
\text{I}_{1} \leq 2 \int_{\tilde{\Omega}}\int_{\tilde{\Omega}\cap\{|y-x|\geq 1\}}\frac{|u(y)|^{2}+|u(x)|^{2}}{|y-x|^{n+2\beta_{*}}}dydx
\leq C \|u\|_{L^{2}(\tilde{\Omega})}^{2}
\end{align*}
and
\begin{align*}
\text{I}_{2} \leq \int_{\tilde{\Omega}}\int_{\tilde{\Omega}\cap\{|y-x|< 1\}}\frac{|u(y)-u(x)|^{2}}{|y-x|^{n+2\beta(x)}}dydx
\leq \|u\|_{H^{\beta(\cdot)}(\tilde{\Omega})}^{2},
\end{align*}
we infer that
\begin{align*}
\|u\|_{H^{\beta_{*}}(\tilde{\Omega})} \leq C \|u\|_{H^{\beta(\cdot)}(\tilde{\Omega})},
\end{align*}
where (\ref{embeddingGuji1}) has been used when we estimate $\text{I}_{1}$.
Similarly, we could obtain
\begin{align*}
\|u\|_{H^{\beta(\cdot)}(\tilde{\Omega})} \leq C \|u\|_{H^{\beta^{*}}(\tilde{\Omega})}.
\end{align*}
Hence, the proof is completed.
\end{proof}

\begin{lemma}\label{upperEmbedding}
Let the function $\gamma$ defined as in (\ref{defineGamma}). Then
\begin{align*}
|u|_{H^{\beta(\cdot)}(\tilde{\Omega})}^{2} \leq a_{*}^{-1}|||u|||^{2} + C \epsilon^{-2\beta_{*}}\|u\|_{L^{2}(\tilde{\Omega})}^{2},
\end{align*}
and
\begin{align*}
|||u|||^{2} \leq a^{*}|u|_{H^{\beta(\cdot)}(\tilde{\Omega})}^{2}.
\end{align*}
\end{lemma}
\begin{proof}
We have
\begin{align*}
|u|_{H^{\beta(\cdot)}(\tilde{\Omega})}^{2} = &
\int_{\tilde{\Omega}}\int_{\tilde{\Omega}\cap B_{\epsilon}(x)} \frac{|u(y)-u(x)|^{2}}{|y-x|^{n+2\beta(x)}}dydx
+ \int_{\tilde{\Omega}}\int_{\tilde{\Omega}\backslash B_{\epsilon}(x)}\frac{|u(y)-u(x)|^{2}}{|y-x|^{n+2\beta(x)}}dydx \\
\leq & \int_{\tilde{\Omega}}\int_{\tilde{\Omega}\cap B_{\epsilon}(x)} \frac{|u(y)-u(x)|^{2}}{|y-x|^{n+2\beta(x)}}dydx
+ \int_{\tilde{\Omega}}\int_{\tilde{\Omega}\backslash B_{\epsilon}(x)}\frac{|u(y)|^{2}+|u(x)|^{2}}{|y-x|^{n+2\beta(x)}}dydx.
\end{align*}
Noticing
\begin{align*}
\int_{\tilde{\Omega}}\int_{\tilde{\Omega}\backslash B_{\epsilon}(x)}\frac{|u(x)|^{2}}{|y-x|^{n+2\beta(x)}}dydx
\leq & C \int_{\tilde{\Omega}}\int_{\epsilon}^{\infty}\frac{1}{|z|^{n+2\beta_{*}}}dz |u(x)|^{2}dx \\
\leq & C \epsilon^{-2\beta_{*}}\|u\|_{L^{2}(\tilde{\Omega})}^{2},
\end{align*}
and
\begin{align*}
\int_{\tilde{\Omega}}\int_{\tilde{\Omega}\backslash B_{\epsilon}(x)}\frac{|u(y)|^{2}}{|y-x|^{n+2\beta(x)}}dydx
= & \int_{\tilde{\Omega}}\int_{\tilde{\Omega}\backslash B_{\epsilon}(y)} \frac{1}{|y-x|^{n+2\beta(x)}}dx |u(y)|^{2} dy \\
\leq & C \epsilon^{-2\beta_{*}}\|u\|_{L^{2}(\tilde{\Omega})}^{2},
\end{align*}
we conclude that
\begin{align*}
|u|_{H^{\beta(\cdot)}(\tilde{\Omega})}^{2} \leq a_{*}^{-1}|||u|||^{2} + C \epsilon^{-2\beta_{*}}\|u\|_{L^{2}(\tilde{\Omega})}^{2}.
\end{align*}
The second inequality follows directly from
\begin{align*}
\int_{\tilde{\Omega}}\int_{\tilde{\Omega}} \mathcal{D}^{*}u\cdot (a(t,x,y)\cdot \mathcal{D}^{*}u) dydx
\leq a^{*}\int_{\tilde{\Omega}}\int_{\tilde{\Omega}}\frac{|u(y)-u(x)|^{2}}{|y-x|^{n+2\beta(x)}}dydx.
\end{align*}
\end{proof}

\begin{lemma}\label{poincareInequality}
(Poincar\'{e} type inequality)
Let the function $\gamma$ defined as in (\ref{defineGamma}). Then there exists a positive constant $C$ such that
\begin{align*}
\|u\|_{L^{2}(\tilde{\Omega})}^{2} \leq C |||u|||^{2} \quad\quad \forall \, u \in V_{c}(\tilde{\Omega}).
\end{align*}
\end{lemma}
\begin{proof}
From Lemma \ref{embeddingVariableSpace}, we know that $H^{\beta(\cdot)}(\tilde{\Omega})\subset H^{\beta_{*}}(\tilde{\Omega})$.
Then, because $H^{\beta_{*}}(\tilde{\Omega})$ is compactly embedded in $L^{2}(\tilde{\Omega})$,
the embedding $H^{\beta(\cdot)}(\tilde{\Omega}) \hookrightarrow L^{2}(\tilde{\Omega})$ is compact.
With this compact embedding result, this lemma could be proved by standard contradiction arguments.
For the standard contradiction arguments, we refer to \cite{du2012analysis,evans1998partial}.
\end{proof}

From Lemma \ref{upperEmbedding} and Lemma \ref{poincareInequality}, we obtain the following theorem.
\begin{theorem}\label{equiNorm1}
Let the function $\gamma$ defined as in (\ref{defineGamma}). Then there exists a constant $C > 0$ such that
\begin{align*}
C^{-1}\|u\|_{H^{\beta(\cdot)}(\Omega\cup\Omega_{\mathcal{I}})} \leq |||u||| \leq C \|u\|_{H^{\beta(\cdot)}(\Omega\cup\Omega_{\mathcal{I}})},
\quad \forall\, u\in V_{c}(\Omega\cup\Omega_{I}).
\end{align*}
In addition, the constrained spaces $H_{c}^{\beta(\cdot)}(\Omega\cup\Omega_{\mathcal{I}})$ and $V_{c}(\Omega\cup\Omega_{\mathcal{I}})$
are equivalent.
\end{theorem}

With these preparations, and recalling the parabolic variational formulation given in Definition \ref{parabolicVariational},
we could obtain the following theorem by using the Lax-Milgram theorem and the Galerkin-type arguments.
For the Galerkin-type arguments, we refer to \cite{evans1998partial,Felsinger2015}.
Since the arguments are standard, we omit the details here.
\begin{theorem}\label{wellPosedness1}
Let $\beta(\cdot)$ to be a continuous function and there exist $\beta_{*}$ and $\beta^{*}$ such that
$0 < \beta_{*} \leq \beta(x) \leq \beta^{*} < 1$.
Assume that $f\in L^{2}(0,T; (H_{c}^{\beta(\cdot)}(\Omega\cup\Omega_{\mathcal{I}}))^{*})$ and $u_{0}\in L_{c}^{2}(\Omega\cup\Omega_{\mathcal{I}})$,
then the initial volume-constrained problem
(\ref{Equation1}) or (\ref{Equation2}) with nonlocal operator defined by (\ref{defineAlpha}) and (\ref{defineGamma})
has a unique solution
$u\in C(0,T;H_{c}^{\beta(\cdot)}(\Omega\cup\Omega_{\mathcal{I}}))\cap H^{1}(0,T;(H_{c}^{\beta(\cdot)}(\Omega\cup\Omega_{\mathcal{I}}))^{*})$.
\end{theorem}

\subsection{Regularity improvement}

In this subsection, we provide more regularity properties of the solution constructed in Theorem \ref{wellPosedness1}.
Define $L^{2}(0,T; \tilde{H}^{2\beta(\cdot)}(\tilde{\Omega}))$ as a space which includes functions in the following set
\begin{align*}
\left\{ u\in L_{c}^{2}(\tilde{\Omega}) \,:\,
\int_{0}^{T}\int_{\Omega}\left| \int_{\tilde{\Omega}} (u(y,t)-u(x,t))\gamma(t,x,y) dy \right|^{2}dxdt < \infty
\right\}
\end{align*}
with
\begin{align*}
\|u\|_{L^{2}(0,T; \tilde{H}^{2\beta(\cdot)}(\tilde{\Omega}))}^{2} :=
\int_{0}^{T}\int_{\Omega}\left| \int_{\tilde{\Omega}} (u(y,t)-u(x,t))\gamma(t,x,y) dy \right|^{2}dxdt.
\end{align*}

Then we show the main result in the following theorem.
\begin{theorem}\label{regualrityTheorem}
Assume
\begin{align*}
u_{0}\in H_{c}^{\beta(\cdot)}(\tilde{\Omega}), \quad f\in L^{2}(0,T; L^{2}(\tilde{\Omega})).
\end{align*}
Suppose $u$ is the weak solution stated in Theorem \ref{wellPosedness1}, then we have
\begin{align}\label{regularitySol1}
u\in C(0,T;H_{c}^{\beta(\cdot)}(\tilde{\Omega})) \cap L^{2}(0,T; \tilde{H}^{2\beta(\cdot)}(\tilde{\Omega})),
\quad
\partial_{t}u \in L^{2}(0,T; L^{2}(\Omega)),
\end{align}
and in addition, we have the estimate
\begin{align}\label{regularityEs}
\begin{split}
\sup_{0\leq t\leq T}\|u(\cdot,t)\|_{H^{\beta(\cdot)}(\tilde{\Omega})}
+ & \|\partial_{t}u\|_{L^{2}(0,T; L^{2}(\Omega))}
+ \|u\|_{L^{2}(0,T; \tilde{H}^{2\beta(\cdot)}(\tilde{\Omega}))}     \\
\leq & C \left( \|f\|_{L^{2}(0,T;L^{2}(\tilde{\Omega}))} + \|u_{0}\|_{H^{\beta(\cdot)}(\tilde{\Omega})} \right),
\end{split}
\end{align}
where the constant $C$ depending on $T$, $\Omega$ and $a$.
\end{theorem}

By Galerkin approximations, we can prove this theorem rigorously.
However, the proof seems to be standard, so we refer to the proof of Theorem 5 in Chapter 7 in \cite{evans1998partial} as a prototype
and only provide a simple formal derivation for concise.
\begin{proof}
From (\ref{Equation1}) or (\ref{Equation2}), we could obtain
\begin{align*}
\int_{\Omega}f^{2}dx = & \int_{\Omega} \left| \partial_{t}u + \mathcal{D}(a\cdot\mathcal{D}^{*}u) \right|^{2} dx    \\
= & \int_{\Omega}|\partial_{t}u|^{2} + |\mathcal{D}(a\cdot\mathcal{D}^{*}u)|^{2}
+ 2\mathcal{D}(a\cdot\mathcal{D}^{*}u)\cdot\partial_{t}u dx.
\end{align*}
Using the generalized nonlocal Green's first identity (\ref{nonlocalGreenFirst}), we have
\begin{align}\label{regularityIn1}
\begin{split}
\int_{\Omega}f^{2}dx = & \int_{\Omega}|\partial_{t}u|^{2} + |\mathcal{D}(a\cdot\mathcal{D}^{*}u)|^{2}dx
+ 2 \int_{\tilde{\Omega}}\int_{\tilde{\Omega}} \mathcal{D}^{*}\partial_{t}u \cdot a(t,x,y) \cdot \mathcal{D}^{*}udydx.
\end{split}
\end{align}
We have
\begin{align}\label{regularityIn2}
\begin{split}
& 2\int_{\tilde{\Omega}}\int_{\tilde{\Omega}}\partial_{t}(u(y,t)-u(x,t))\gamma(t,x,y)(u(y,t)-u(x,t))dydx \\
= & \frac{d}{dt}\int_{\tilde{\Omega}}\int_{\tilde{\Omega}}(u(y,t)-u(x,t))^{2}\gamma(t,x,y)dydx  \\
& \quad\quad\quad\quad\quad\quad\quad
- \int_{\tilde{\Omega}}\int_{\tilde{\Omega}}(u(y,t)-u(x,t))^{2}\partial_{t}\gamma(t,x,y)dydx.
\end{split}
\end{align}
Integrating (\ref{regularityIn1}) from $0$ to $T$, and using (\ref{regularityIn2}), we could obtain
\begin{align*}
& \int_{0}^{T}\int_{\Omega}|\partial_{t}u|^{2} + |\mathcal{D}(a\cdot\mathcal{D}^{*}u)|^{2}dxdt
+ \int_{\tilde{\Omega}}\int_{\tilde{\Omega}}(u(y,t)-u(x,t))^{2}\gamma(t,x,y)dydx    \\
& \leq \int_{0}^{T}\int_{\Omega}f^{2}dxdt +
\int_{\tilde{\Omega}}\int_{\tilde{\Omega}}(u(y,0)-u(x,0))^{2}\gamma(0,x,y)dydx  \\
& \quad\quad\quad\quad\quad\quad\quad\,\,\,
+ \int_{0}^{T}\int_{\tilde{\Omega}}\int_{\tilde{\Omega}}(u(y,t)-u(x,t))^{2}\partial_{t}\gamma(t,x,y)dydxdt
\end{align*}
From the above estimate and remembering the estimate of weak solutions, we obtain the desired result.
\end{proof}

\section{Carleman type estimate}\label{CarlemanEstimateSection}

In this section, we focus on a Carleman type estimate, which is an important tool for researches about inverse problems.
We denote
\begin{align}\label{LLLL}
L(u) = \partial_{t}u + \mathcal{D}(a\cdot\mathcal{D}^{*}u).
\end{align}
In order to obtain a Carleman estimate, we adopt the following weight function
\begin{align*}
\varphi(t) = e^{\lambda t}
\end{align*}
where $\lambda > 0$ is fixed suitably.
This weight function has been used in \cite{yamamoto2009carleman} to obtain a Carleman type estimate for some
integer order diffusion equations. Since our spatial derivative is an anisotropic variable-order fractional Laplace operator,
we use the nonlocal vector calculus compared with the integer order diffusion equation case.

\begin{theorem}\label{carlemanEstimateThe}
(Carleman estimate) We set $\varphi(t) = e^{\lambda t}$. Then there exists $\lambda_{0} > 0$
such that for any arbitrary $\lambda \geq \lambda_{0}$ we can choose a constant $s_{0}(\lambda) > 0$ satisfying:
there exists a constant $C = C(s_{0},\lambda_{0}) > 0$ such that
\begin{align*}
& \int_{Q}\left\{ \frac{1}{s\varphi}\left( |\partial_{t}u|^{2} + |\mathcal{D}(a\cdot\mathcal{D}^{*}u)|^{2} \right)
+ s\lambda^{2}\varphi u^{2} \right\}e^{2s\varphi}dxdt   \\
& \quad\quad\quad\quad
+ \lambda \int_{0}^{T}\int_{\tilde{\Omega}}\int_{\tilde{\Omega}}\frac{|v(y,t)-v(x,t)|^{2}}{|y-x|^{n+2\beta(x)}}e^{2s\varphi}dydxdt \\
& \quad\quad\quad\quad
\leq C \int_{\tilde{Q}}|L(u)|^{2}e^{2s\varphi}dxdt
+ C e^{C(\lambda)s}\left( \|u(\cdot,T)\|_{H^{\beta(\cdot)}(\tilde{\Omega})}^{2} + \|u(\cdot,0)\|_{H^{\beta(\cdot)}(\tilde{\Omega})}^{2} \right)
\end{align*}
for all $s > s_{0}$ and all
$u\in C(0,T;H^{\beta(\cdot)}(\tilde{\Omega}))\cap H^{1}(0,T;L^{2}(\Omega))\cap L^{2}(0,T; \tilde{H}^{2\beta(\cdot)}(\tilde{\Omega}))$ satisfying
\begin{align}\label{carelThe1}
u = 0 \quad \text{in }\Omega_{\mathcal{I}}\times (0,T)
\end{align}
or
\begin{align}\label{carelThe2}
\mathcal{N}(a\cdot\mathcal{D}^{*}u) = 0 \quad \text{in }\Omega_{\mathcal{I}}\times (0,T).
\end{align}
\end{theorem}
\begin{proof}
Set
\begin{align*}
v = e^{s\varphi}u, \quad\quad Pv = e^{s\varphi}L(e^{-s\varphi}v) = e^{s\varphi}f.
\end{align*}
Assume that $u|_{\Omega_{\mathcal{I}}} = 0$ or $\mathcal{N}(a\cdot\mathcal{D}^{*}u)|_{\Omega_{\mathcal{I}}} = 0$.
Obviously, we obtain
\begin{align}\label{carest1}
Pv = \partial_{t}v - (s\lambda\varphi v - \mathcal{D}(a\cdot \mathcal{D}^{*}v)) = e^{s\varphi}f.
\end{align}
In addition, we have
\begin{align}
\|e^{s\varphi}f\|_{L^{2}(Q)}^{2} = & \int_{Q}|\partial_{t}v|^{2}dxdt +
2\int_{Q}\partial_{t}v \left( -s\lambda\varphi v + \mathcal{D}(a\cdot\mathcal{D}^{*}v) \right)dxdt \nonumber  \\
& + \int_{Q} |s\lambda\varphi v - \mathcal{D}(a\cdot\mathcal{D}^{*}v)|^{2}dxdt  \nonumber  \\
\geq & \int_{Q}|\partial_{t}v|^{2}dxdt + 2\int_{Q}\partial_{t}v \mathcal{D}(a\cdot\mathcal{D}^{*}v) dxdt
 + 2\int_{Q}\partial_{t}v (-s\lambda\varphi v)dxdt   \nonumber \\
\equiv & \int_{Q}|\partial_{t}v|^{2}dxdt + \text{I}_{1} + \text{I}_{2}.
\end{align}
Thus
\begin{align}\label{carest2}
\int_{Q}f^{2}e^{2s\varphi}dxdt \geq \text{I}_{1} + \text{I}_{2}
\end{align}
and
\begin{align}\label{carest13}
\int_{Q}|\partial_{t}v|^{2}dxdt \leq \int_{Q}f^{2}e^{2s\varphi}dxdt + |\text{I}_{1} + \text{I}_{2}|.
\end{align}
In the following, $C_{j}>0$ ($j\in\mathbb{N}$) denote generic constants which are independent of $s,\lambda$.
Since $s,\lambda$ are assumed to be a large enough constant, without loss of generality, we can assume $s>1$ and $\lambda >1$.
For $\text{I}_{1}$, we have
\begin{align}\label{carest3}
|\text{I}_{1}| \geq & \left| 2\int_{0}^{T}\int_{\Omega}\partial_{t}v \mathcal{D}(a\cdot\mathcal{D}^{*}v)dxdt \right| \nonumber \\
= & \left| 2\int_{0}^{T}\int_{\tilde{\Omega}}\int_{\tilde{\Omega}} \mathcal{D}^{*}\partial_{t}v \cdot
a\cdot \mathcal{D}^{*}v dydxdt + 2\int_{0}^{T}\int_{\Omega_{\mathcal{I}}}\partial_{t}v \mathcal{N}(a\cdot\mathcal{D}^{*}v)dxdt \right|
\nonumber \\
= & \left| 2\int_{0}^{T}\int_{\tilde{\Omega}}\int_{\tilde{\Omega}} (\partial_{t}v(y,t)-\partial_{t}v(x,t))
\gamma(t,x,y) (v(y,t) - v(x,t)) dydxdt \right| \nonumber \\
= & \left| \int_{0}^{T}\int_{\tilde{\Omega}}\int_{\tilde{\Omega}}
\partial_{t}\left[(v(y,t)-v(x,t))^{2}\right] \gamma(t,x,y) dydxdt \right| \nonumber \\
= & \Bigg| -\int_{0}^{T}\int_{\tilde{\Omega}}\int_{\tilde{\Omega}}\partial_{t}\gamma(t,x,y) (v(y,t)-v(x,t))^{2}  dydxdt \nonumber \\
& \quad\quad\quad\quad\quad\quad\quad\quad\quad
+ \int_{\tilde{\Omega}}\int_{\tilde{\Omega}} \gamma(t,x,y)(v(y,t)-v(x,t))^{2}dydx \Big|_{t=0}^{t=T} \Bigg|  \nonumber \\
\leq & C_{1}\|v\|_{L^{2}(0,T;H^{\beta(\cdot)}(\tilde{\Omega}))}^{2} + C_{1}\|v(\cdot,T)\|_{H^{\beta(\cdot)}(\tilde{\Omega})}^{2}
+ C_{1}\|v(\cdot,0)\|_{H^{\beta(\cdot)}(\tilde{\Omega})}^{2},
\end{align}
where (\ref{aAssumption1}) has been used for the last inequality. For $\text{I}_{2}$, we have
\begin{align}\label{carest4}
\begin{split}
\text{I}_{2} = & -s\lambda \int_{Q}2\partial_{t}v \, v \, \varphi dxdt = s\lambda \int_{Q}v^{2}\partial_{t}\varphi dxdt
- s\lambda \Big( \int_{\Omega}\varphi v^{2}dx\Big) \Big|_{t=0}^{t=T}    \\
\geq & s\lambda^{2}\int_{Q}\varphi v^{2}dxdt - s\lambda \int_{\Omega}
\left( e^{\lambda T}|v(x,T)|^{2} + |v(x,0)|^{2} \right)dx.
\end{split}
\end{align}
From (\ref{carest2}), (\ref{carest3}) and (\ref{carest4}), we obtain
\begin{align}\label{carest5}
\begin{split}
\|e^{s\varphi}f\|_{L^{2}(Q)}^{2} \geq & s\lambda^{2}\int_{Q}\varphi v^{2}dxdt - C_{1}\|v\|_{L^{2}(0,T;H^{\beta(\cdot)}(\tilde{\Omega}))}^{2}
- C_{1}\|v(\cdot,T)\|_{H^{\beta(\cdot)}(\tilde{\Omega})}^{2}    \\
& - C_{1}\|v(\cdot,0)\|_{H^{\beta(\cdot)}(\tilde{\Omega})}^{2}
- s\lambda\int_{\tilde{\Omega}}\left( e^{\lambda T}|v(x,T)|^{2} + |v(x,0)|^{2} \right)dx.
\end{split}
\end{align}
In the following, we estimate $\|v\|_{L^{2}(0,T;H^{\beta(\cdot)}(\tilde{\Omega}))}^{2}$.
We have
\begin{align}\label{carest6}
\begin{split}
\int_{Q}(Pv)v dxdt = & \int_{Q}v\partial_{t}vdxdt - \int_{Q}s\lambda \varphi v^{2}dxdt
+ \int_{Q}v \mathcal{D}(a\cdot\mathcal{D}^{*}v)dxdt \\
\equiv & J_{1} + J_{2} + J_{3}.
\end{split}
\end{align}
For $J_{1}$, we obtain
\begin{align*}
|J_{1}| = & \left| \int_{Q}v\partial_{t}v dxdt \right| = \left| \frac{1}{2}\int_{Q}\partial_{t}(v^{2}) dxdt \right|
\leq \frac{1}{2}\int_{\tilde{\Omega}}\left( |v(x,T)|^{2} + |v(x,0)|^{2} \right)dx.
\end{align*}
For $J_{2}$, we have
\begin{align*}
|J_{2}| = \left| -\int_{Q}s\lambda\varphi v^{2}dxdt \right| \leq C_{2} \int_{Q} s\lambda\varphi v^{2}dxdt.
\end{align*}
At last, for $J_{3}$, we have
\begin{align*}
J_{3} = & \int_{0}^{T}\int_{Q}v\cdot\mathcal{D}(a\cdot\mathcal{D}^{*}v) dxdt
= \int_{0}^{T}\int_{\tilde{\Omega}}\int_{\tilde{\Omega}}\mathcal{D}^{*}v \cdot a(t,x,y) \cdot \mathcal{D}^{*}v dydxdt \\
= & \int_{0}^{T}\int_{\tilde{\Omega}}\int_{\tilde{\Omega}} (v(y,t)-v(x,t))^{2}\gamma(t,x,y)dydxdt   \\
\geq & a_{*}\int_{0}^{T}\int_{\tilde{\Omega}}\int_{\tilde{\Omega}} \frac{(v(y,t)-v(x,t))^{2}}{|y-x|^{n+2\beta(x)}}dydxdt.
\end{align*}
From (\ref{carest6}) and the above estimates on $J_{1}$, $J_{2}$ and $J_{3}$, we obtain
\begin{align}\label{carest7}
\begin{split}
\int_{Q}\lambda (Pv)v dxdt \geq &
a_{*}\int_{0}^{T}\int_{\tilde{\Omega}}\int_{\tilde{\Omega}} \lambda\frac{(v(y,t)-v(x,t))^{2}}{|y-x|^{n+2\beta(x)}}dydxdt
- C_{2}\int_{Q}s\lambda^{2}\varphi v^{2}dxdt    \\
& - \frac{1}{2}\lambda\int_{\tilde{\Omega}}\left( |v(x,T)|^{2} + |v(x,0)|^{2} \right)dx.
\end{split}
\end{align}
On the other hand,
\begin{align}\label{carest8}
\begin{split}
\int_{Q}\lambda (Pv)v dxdt \leq & \|Pv\|_{L^{2}(Q)} \left(\lambda \|v\|_{L^{2}(Q)} \right)
\leq \frac{1}{2}\|Pv\|_{L^{2}(Q)}^{2} + \frac{\lambda^{2}}{2}\|v\|_{L^{2}(Q)}^{2}   \\
\leq & \frac{1}{2}\|fe^{s\varphi}\|_{L^{2}(Q)}^{2} + \frac{\lambda^{2}}{2}\|v\|_{L^{2}(Q)}^{2}.
\end{split}
\end{align}
Hence, (\ref{carest7}) and (\ref{carest8}) yield
\begin{align*}
a_{*}\int_{0}^{T}\int_{\tilde{\Omega}}\int_{\tilde{\Omega}}\lambda\frac{(v(y,t)-v(x,t))^{2}}{|y-x|^{n+2\beta(x)}}dydxdt
\leq & C_{2}\int_{Q}s\lambda^{2}\varphi v^{2}dxdt + \frac{1}{2}\|fe^{s\varphi}\|_{L^{2}(\tilde{Q})}^{2} \\
+ \frac{\lambda^{2}}{2}\|v\|_{L^{2}(Q)}^{2} & +
\frac{1}{2}\lambda\int_{\tilde{\Omega}} \left( |v(x,T)|^{2} + |v(x,0)|^{2} \right)dx.
\end{align*}
Estimating the first term on the right-hand side by (\ref{carest5}), we obtain
\begin{align}\label{carest9}
a_{*}\int_{0}^{T}\int_{\tilde{\Omega}}\int_{\tilde{\Omega}}\lambda & \frac{(v(y,t)-v(x,t))^{2}}{|y-x|^{n+2\beta(x)}} dydxdt
\leq C_{3}\|fe^{s\varphi}\|_{L^{2}(\tilde{Q})}^{2} + C_{3}\|v\|_{L^{2}(0,T; H^{\beta(\cdot)}(\tilde{\Omega}))}^{2} \nonumber \\
& \quad\quad\quad\quad
+ C_{3}\lambda^{2}\|v\|_{L^{2}(Q)}^{2} + C_{3}\lambda \left( \|v(\cdot,T)\|_{L^{2}(\tilde{\Omega})}^{2}
+ \|v(\cdot,0)\|_{L^{2}(\tilde{\Omega})}^{2} \right)   \nonumber  \\
& \quad\quad\quad\quad
+ C_{3}(\|v(\cdot,T)\|_{H^{\beta(\cdot)}(\tilde{\Omega})}^{2} + \|v(\cdot,0)\|_{H^{\beta(\cdot)}(\tilde{\Omega})}^{2} ) \nonumber \\
& \quad\quad\quad\quad
+ C_{3}s\lambda (e^{\lambda T}\|v(\cdot,T)\|_{L^{2}(\tilde{\Omega})}^{2} + \|v(\cdot,0)\|_{L^{2}(\tilde{\Omega})}^{2}).
\end{align}
Considering both (\ref{carest5}) and (\ref{carest9}), we obtain
\begin{align}\label{carest10}
s\lambda^{2}\int_{Q}\varphi v^{2} dxdt \, + \, &
a_{*}\int_{0}^{T}\int_{\tilde{\Omega}}\int_{\tilde{\Omega}} \lambda\frac{|v(y,t)-v(x,t)|^{2}}{|y-x|^{n+2\beta(x)}}dydxdt  \nonumber \\
\leq & C_{4}\|fe^{s\varphi}\|_{L^{2}(\tilde{Q})}^{2} + C_{4}\|v\|_{H^{\beta(\cdot)}(\tilde{Q})}^{2}
+ C_{4}\lambda^{2}\|v\|_{L^{2}(Q)}^{2}     \nonumber \\
& + C_{4}\left(\|v(\cdot,T)\|_{H^{\beta(\cdot)}(\tilde{\Omega})}^{2} + \|v(\cdot,0)\|_{H^{\beta(\cdot)}(\tilde{\Omega})}^{2}\right) \nonumber \\
& + C_{4}s\lambda \left(e^{\lambda T}\|v(\cdot,T)\|_{L^{2}(\tilde{\Omega})}^{2} + \|v(\cdot,0)\|_{L^{2}(\tilde{\Omega})}^{2}\right).
\end{align}
Now, we take $s>0$, $\lambda>0$ large to absorb the second and third terms on the right-hand side into the left-hand side, then
we obtain
\begin{align*}
\int_{Q}s\lambda^{2}\varphi v^{2}dxdt + &
\int_{0}^{T}\int_{\tilde{\Omega}}\int_{\tilde{\Omega}}\lambda\frac{|v(y,t)-v(x,t)|^{2}}{|y-x|^{n+2\beta(x)}}dydxdt \\
\leq & C_{5}\|fe^{s\varphi}\|_{L^{2}(\tilde{Q})}^{2}
+ C_{5}e^{C(\lambda)s}\left( \|v(\cdot,T)\|_{H^{\beta(\cdot)}(\tilde{\Omega})}^{2} + \|v(\cdot,0)\|_{H^{\beta(\cdot)}(\tilde{\Omega})}^{2} \right).
\end{align*}
Since $v = e^{s\varphi}u$, in addition, we have
\begin{align}\label{carest11}
\int_{Q}s\lambda^{2}\varphi u^{2}e^{2s\varphi} & dxdt +
\int_{0}^{T}\int_{\tilde{\Omega}}\int_{\tilde{\Omega}}\lambda\frac{|u(y,t)-u(x,t)|^{2}}{|y-x|^{n+2\beta(x)}}e^{2s\varphi}dydxdt \nonumber \\
\leq & C_{5}\|fe^{s\varphi}\|_{L^{2}(\tilde{Q})}^{2} +
C_{5}e^{C(\lambda)s}\left( \|u(\cdot,T)\|_{H^{\beta(\cdot)}(\tilde{\Omega})}^{2} + \|u(\cdot,0)\|_{H^{\beta(\cdot)}(\tilde{\Omega})}^{2} \right).
\end{align}
Since $\partial_{t}u = -s\lambda \varphi e^{-s\varphi} v + e^{-s\varphi}\partial_{t}v$, we obtain
\begin{align}\label{carest12}
\frac{1}{s\varphi}|\partial_{t}u|^{2}e^{2s\varphi} \leq 2s\lambda^{2}\varphi v^{2} + \frac{2}{s\varphi}|\partial_{t}v|^{2}.
\end{align}
By (\ref{carest13}), (\ref{carest11}) and estimates about $\text{I}_{1}, \text{I}_{2}$, we find
\begin{align}\label{carest14}
\begin{split}
\int_{Q}\frac{1}{s\varphi}|\partial_{t}u|^{2}e^{2s\varphi}dxdt \leq &
C \int_{\tilde{Q}}f^{2}e^{2s\varphi}dxdt  \\
& + C e^{C(\lambda)s}\left( \|u(\cdot,T)\|_{H^{\beta(\cdot)}(\tilde{\Omega})}^{2} + \|u(\cdot,0)\|_{H^{\beta(\cdot)}(\tilde{\Omega})}^{2} \right).
\end{split}
\end{align}
From $\mathcal{D}(a\cdot\mathcal{D}^{*}u) = f - \partial_{t}u$, we could finish the proof by using (\ref{carest11}) and (\ref{carest14}).
\end{proof}

\begin{remark}\label{ynzhengtiaojian}
From Theorem \ref{regualrityTheorem}, we know that the regularity condition required in Theorem \ref{carlemanEstimateThe}
could be satisfied.
\end{remark}

\begin{remark}\label{lowerOrderTerm}
From Section 5.3 in \cite{du2012analysis}, we know that a variable-order space fractional advection-diffusion equation could be formulated as follow
\begin{align}\label{variSpace1}
\begin{split}
\partial_{t}u + \mathcal{D}(a\cdot\mathcal{D}^{*}u) + \frac{1}{2}\mathcal{D}(\mu u) = f \quad \forall \, x\in \Omega, \, t>0.
\end{split}
\end{align}
And we have
\begin{align*}
\mathcal{D}(\mu u) = u\mathcal{D}(\mu) - \int_{\mathbb{R}^{n}}\mu\cdot\mathcal{D}^{*}(u)dy.
\end{align*}
Since
\begin{align*}
& \int_{\tilde{\Omega}}\int_{\tilde{\Omega}}-\mu\frac{(u(y,t)-u(x,t))(y-x)}{|y-x|^{n/2+\beta(x)+1}}dydx \\
& \quad\quad\quad\quad
\leq \left(\int_{\tilde{\Omega}}\int_{\tilde{\Omega}}\mu^{2}dydx\right)^{1/2}
\left( \int_{\tilde{\Omega}}\int_{\tilde{\Omega}}\frac{|u(y,t)-u(x,t)|^{2}}{|y-x|^{n+2\beta(x)}}dydx \right)^{1/2},
\end{align*}
by some simple calculations, we can easily see that if we take $s,\lambda$ large enough, a Carleman estimate also holds for
the variable-order space fractional advection-diffusion equation
as illustrated for the integer-order diffusion equations in \cite{yamamoto2009carleman}.
\end{remark}

In the last part of this section, we show the following theorem which could be proved similar to Theorem \ref{carlemanEstimateThe}.
\begin{theorem}\label{carlemanEstimateThe11}
We set $\varphi(t) = e^{-\lambda t}$. Then there exists $\lambda_{0} > 0$
such that for any arbitrary $\lambda \geq \lambda_{0}$ we can choose a constant $s_{0}(\lambda) > 0$ satisfying:
there exists a constant $C = C(s_{0},\lambda_{0}) > 0$ such that
\begin{align*}
& \int_{Q}\left\{ \frac{1}{s\varphi}\left( |\partial_{t}u|^{2} + |\mathcal{D}(a\cdot\mathcal{D}^{*}u)|^{2} \right)
+ s\lambda^{2}\varphi u^{2} \right\}e^{2s\varphi}dxdt   \\
& \quad\quad\quad\quad
+ \lambda \int_{0}^{T}\int_{\tilde{\Omega}}\int_{\tilde{\Omega}}\frac{|v(y,t)-v(x,t)|^{2}}{|y-x|^{n+2\beta(x)}}e^{2s\varphi}dydxdt \\
& \quad\quad\quad\quad
\leq C \int_{\tilde{Q}}|L(u)|^{2}e^{2s\varphi}dxdt
+ C e^{C(\lambda)s}\|u(\cdot,0)\|_{H^{\beta(\cdot)}(\tilde{\Omega})}^{2} \\
& \quad\quad\quad\quad\quad
+ C \int_{\Omega_{\mathcal{I}}\times(0,T)}s\lambda (|u|+|\partial_{t}u|)|\mathcal{N}(a\cdot\mathcal{D}^{*}u)|e^{2s\varphi}dxdt
\end{align*}
for all $s > s_{0}$ and all
$u\in C(0,T;H^{\beta(\cdot)}(\tilde{\Omega}))\cap H^{1}(0,T;L^{2}(\Omega))\cap L^{2}(0,T; \tilde{H}^{2\beta(\cdot)}(\tilde{\Omega}))$ satisfying
\begin{align}\label{carelThe11}
u(\cdot,T) = 0 \quad \text{in }\tilde{\Omega}.
\end{align}
\end{theorem}

\section{Applications to two inverse problems}\label{BackwardSection}

In the first part this section, we apply Theorem \ref{carlemanEstimateThe} to a backward diffusion problem and establish a
conditional stability estimate.
Let us now specify the problem as follows.

\emph{Backward in time problem}: Let $0 \leq t_{0} < T$. For system (\ref{Equation1}) or (\ref{Equation2}), determine $u(x,t_{0})$, $x\in \Omega$ from $u(x,T)$, $x\in \Omega\cup\Omega_{\mathcal{I}}$.

For this backward in time problem, we have the following theorem.
\begin{theorem}\label{backwardTheorem}
Let $u$ to be a solution of system (\ref{Equation1}) or (\ref{Equation2}) satisfying
\begin{align*}
u\in C(0,T;H_{c}^{\beta(\cdot)}(\tilde{\Omega})) \cap L^{2}(0,T; \tilde{H}^{2\beta(\cdot)}(\tilde{\Omega})),
\quad
\partial_{t}u \in L^{2}(0,T; L^{2}(\Omega)).
\end{align*}
For any $t_{0}\in (0,T)$, there exist constants $\theta\in(0,1)$ and $C > 0$ depending on $t_{0}$,
$a_{*}$, $a^{*}$, $T$, $\Omega$ and $\Omega_{\mathcal{I}}$ such that
\begin{align}\label{backTheorem1}
\|u(\cdot,t_{0})\|_{L^{2}(\Omega)} \leq C
\|u\|_{L^{2}(\tilde{Q})}^{1-\theta}\|u(\cdot,T)\|_{H^{\beta(\cdot)}(\Omega\cup\Omega_{\mathcal{I}})}^{\theta}.
\end{align}
If in addition, $\|u(\cdot,0)\|_{L^{2}_{c}(\Omega\cup\Omega_{\mathcal{I}})} < \infty$, then we have
\begin{align}\label{backTheorem2}
\|u\|_{L^{2}(Q)} = O\left( \log \frac{1}{\|u(\cdot,T)\|_{H^{\beta(\cdot)}(\Omega\cup\Omega_{\mathcal{I}})}} \right)^{-1/2}
\end{align}
as $\|u(\cdot,T)\|_{H^{\beta(\cdot)}(\Omega\cup\Omega_{\mathcal{I}})} \rightarrow 0$.
\end{theorem}
\begin{proof}
As in the proof of Theorem 9.2 in \cite{yamamoto2009carleman},
we could choose same cut-off function and by using Theorem \ref{carlemanEstimateThe} to conclude that
\begin{align}\label{yingyong1}
\|u(\cdot,t_{0})\|_{L^{2}(\Omega)} \leq C \|u\|_{L^{2}(\tilde{Q})}^{\frac{C}{C+2(\delta_{0}-\delta_{1})}}
\|u(\cdot,T)\|_{H^{\beta(\cdot)}(\tilde{\Omega})}^{\frac{2(\delta_{0}-\delta_{1})}{C+2(\delta_{0}-\delta_{1})}},
\end{align}
where $\delta_{k}=e^{\lambda t_{k}}$, $k=0,1,2$ and $0<t_{2}<t_{1}<t_{0}$.
Now, integrating the above inequality with respect to $t_{0}$ from $0$ to $T$, we find that
\begin{align*}
\|u\|_{L^{2}(Q)} \leq C(1+\|u(\cdot,0)\|_{L_{c}^{2}(\tilde{\Omega})})
\left( \log \frac{1}{\|u(\cdot,T)\|_{H^{\beta(\cdot)}(\Omega\cup\Omega_{\mathcal{I}})}} \right)^{-1/2}.
\end{align*}
Thus the proof is completed.
\end{proof}

In the second part of this section, let us focus on a special inverse source problem.
Let $x = (x_{1},x')\in\mathbb{R}^{n}$ and $x'=(x_{2},\cdots,x_{n})\in\mathbb{R}^{n-1}$,
$\Omega = (0,\ell)\times D'$, $D'\subset\mathbb{R}^{n-1}$ be a bounded domain with smooth boundary $\partial D'$.
We consider
\begin{align}\label{yingyong21}
\begin{split}
\begin{cases}
\partial_{t}u(x,t) + \mathcal{D}(a\cdot\mathcal{D}^{*}u)(x,t) = f(x',t), \quad x\in \Omega,\,t>0 \\
u(x,0) = 0, \quad x\in\Omega\cup\Omega_{\mathcal{I}}  \\
\mathcal{N}(a\cdot\mathcal{D}^{*}u)(x,t) = 0, \quad x\in\Omega_{\mathcal{I}}, \, t>0.
\end{cases}
\end{split}
\end{align}
Here, we only consider a simple case that is $\beta$ is a constant between $0$ and $1$,
$\epsilon$ appeared in the definition of $\mathcal{D}(a\cdot\mathcal{D}^{*}\cdot)$ is equal to $\infty$
and function $a(t,x,y) \equiv 1$.

\emph{Inverse heat source problem}: Let $t_{0} > 0$. Determine $f(x',t)$, $x'\in D'$, $0<t<t_{0}$ by $u|_{\Omega_{\mathcal{I}}\times(0,t_{0})}$.

For this problem, we have the following result.
\begin{theorem}\label{uniqueness}
We assume that $u, \partial_{x_{1}}u \in C(0,T;H_{c}^{\beta(\cdot)}(\tilde{\Omega})) \cap H^{1}(0,T; L^{2}(\Omega))$,
For $t_{0} > 0$, if $u|_{(\mathbb{R}^{n}\backslash\Omega)\times(0,t_{0})} = 0$, then
$f(x',t) = 0$, $x'\in D'$, $0\leq t\leq t_{0}$.
\end{theorem}
\begin{proof}
For arbitrary small $\epsilon > 0$, we choose $t_{1}$, $t_{2}$ such that $0 < t_{0}-\epsilon <t_{1} < t_{2} < t_{0}$.
We set $\delta_{k} = e^{-\lambda t_{k}}$, $k=0,1,2$. Let $\chi\in C^{\infty}(\mathbb{R})$ be a cut-off function such
that $0\leq\chi\leq 1$ and
\begin{align*}
\chi(t) = \left\{\begin{aligned}
& 1, \quad t < t_{1}, \\
& 0, \quad t \geq t_{2}.
\end{aligned}\right.
\end{align*}
For $u$, we have
\begin{align}\label{inverseSource1}
\partial_{t}u + \mathcal{D}\mathcal{D}^{*}u = f(x',t), \quad x\in\Omega,\, 0<t<t_{2}
\end{align}
and
\begin{align*}
u|_{(\mathbb{R}^{n}\backslash\Omega)\times(0,t_{0})} = \mathcal{N}(\mathcal{D}^{*}u)|_{(\mathbb{R}^{n}\backslash\Omega)\times(0,t_{0})} = 0.
\end{align*}
Differentiating both sides of (\ref{inverseSource1}) with respect to $x_{1}$ and setting $v = \partial_{x_{1}}u$, we obtain
\begin{align}\label{inverseSource2}
\partial_{t}v + \mathcal{D}\mathcal{D}^{*}v = 0, \quad x\in\Omega,\, 0<t<t_{2}.
\end{align}
In the above calculation, we used the property $\partial_{x_{1}}\mathcal{D}\mathcal{D}^{*} = \mathcal{D}\mathcal{D}^{*}\partial_{x_{1}}$,
which relies on our assumptions stated below (\ref{yingyong21}).
Setting $w = \chi v$, we have
\begin{align}\label{inverseSource3}
\partial_{t}w + \mathcal{D}\mathcal{D}^{*}w = - \chi'(t)v,  \quad x\in\Omega,\, 0<t<t_{2}
\end{align}
and
\begin{align*}
w(x,t_{0}) = w(x,0) = 0, \quad w|_{(\mathbb{R}^{n}\backslash\Omega)\times(0,t_{0})} = 0.
\end{align*}
Using Theorem \ref{carlemanEstimateThe11}, we find that
\begin{align*}
\int_{0}^{t_{0}}\int_{\Omega}s\lambda^{2}\varphi w^{2}e^{2s\varphi} dxdt & +
\lambda\int_{0}^{t_{0}}\int_{\tilde{\Omega}}\int_{\tilde{\Omega}}\frac{|w(y,t)-w(x,t)|^{2}}{|y-x|^{n+2\beta}}e^{2s\varphi}dydxdt \\
& \leq C \int_{0}^{t_{0}}\int_{\Omega}|\chi' v|^{2}e^{2s\varphi}dxdt.
\end{align*}
Denote $M = \|\partial_{x_{1}}u\|_{L^{2}(Q)}$. Using the properties of $\chi$, we obtain
\begin{align}\label{inverseSource4}
\int_{0}^{t_{0}}\int_{\Omega}|\chi' v|^{2}e^{2s\varphi}dxdt \leq C M^{2}e^{2s\delta_{1}}.
\end{align}
Setting $\delta_{\epsilon} = e^{-\lambda(t_{0}-\epsilon)}$, using the properties of $\chi$ and (\ref{inverseSource4}), there holds
\begin{align*}
& e^{2s\delta_{\epsilon}}\left\{ \int_{0}^{t_{0}-\epsilon}\int_{\Omega}sw^{2}dxdt
+ \int_{0}^{t_{0}-\epsilon}\int_{\tilde{\Omega}}\int_{\tilde{\Omega}}\frac{|w(y,t)-w(x,t)|^{2}}{|y-x|^{n+2\beta}}dydxdt \right\} \\
\leq & \int_{0}^{t_{0}}\int_{\Omega}sw^{2}e^{2s\varphi}dxdt +
\int_{0}^{t_{0}}\int_{\tilde{\Omega}}\int_{\tilde{\Omega}}\frac{|w(y,t)-w(x,t)|^{2}}{|y-x|^{n+2\beta}}e^{2s\varphi}dydxdt   \\
\leq & C M^{2} e^{2s\delta_{1}}.
\end{align*}
Hence, we easily obtain
\begin{align}\label{inverseSource5}
\|w\|_{L^{2}(0,t_{0}-\epsilon; H^{\beta}(\Omega))} \leq C M^{2} e^{-2s(\delta_{\epsilon}-\delta_{1})}
\end{align}
for all large $s > 0$. Since $\delta_{\epsilon} - \delta_{1} > 0$, letting $s \rightarrow \infty$, we
obtain $w = 0$ in $\Omega\times(0,t_{0}-\epsilon)$. Obviously, we have $v = 0$ in $\Omega\times(0,t_{0}-\epsilon)$.
Considering the assumptions, we have $u = 0$ in $\Omega\times(0,t_{0}-\epsilon)$, which implies $f = 0$ in $D'\times(0,t_{0}-\epsilon)$
by noticing (\ref{inverseSource1}). Because $\epsilon > 0$ is arbitrary, the proof is completed.
\end{proof}


\section*{Acknowledgments}
This work was partially supported by NSFC under Contact 11501439, 11131006 and 41390454,
the Postdoctoral Science Foundation Project of China under grant no. 2015M580826 and
the Natural Science Foundation Project of Shannxi under grant no. 2016JQ1020.


\bibliographystyle{plain}
\bibliography{references}

\end{document}